\documentclass[11pt, reqno]{article}
\usepackage{amssymb,xypic,latexsym}

\usepackage[mathscr]{eucal}
\usepackage{amsmath}
\usepackage{amsthm}
\usepackage[all]{xy}

\theoremstyle{plain}

\newtheorem{corollary}[equation]{Corollary}
\newtheorem{lemma}[equation]{Lemma}
\newtheorem{proposition}[equation]{Proposition}
\newtheorem{theorem}[equation]{Theorem}

\theoremstyle{remark}

\newtheorem{remark}[equation]{Remark}
\newtheorem{remarks}[equation]{Remarks}

\numberwithin{equation}{subsection}

\makeatletter
\renewcommand{\subsection}{\@startsection{subsection}{2}{0pt}{-3ex
plus -1ex minus -0.2ex}{-2mm plus -0pt minus
-2pt}{\normalfont\bfseries}} \makeatother

\newcommand{\scr}[1]{\mathscr{#1}}

\newcommand{\dis}{\displaystyle}

\newcommand{\beq}{\begin{equation}\label}
\newcommand{\eeq}{\end{equation}}

\newcommand{\hdot}{{\:\raisebox{2pt}{\text{\circle*{0.7}}}}}
\newcommand{\idot}{{\:\raisebox{2pt}{\text{\circle*{0.7}}}}}

\newcommand{\erem}{\hphantom{.}\hfill$\lozenge$\end{remark}}

\DeclareMathOperator{\ggr}{\mathrm{gr}}
\DeclareMathOperator{\Lie}{\mathrm{Lie}}
\DeclareMathOperator{\ann}{\mathrm{Ann}}
\DeclareMathOperator{\ddim}{\mathrm{dim}}
\DeclareMathOperator{\Specm}{\mathrm{Spec}}

\DeclareMathOperator{\ad}{\mathrm{ad}}
\DeclareMathOperator{\Ad}{\mathrm{Ad}}
\DeclareMathOperator{\Sym}{{\mathcal{S}{ym}}}
\DeclareMathOperator{\rk}{\mathrm{rk}}
\DeclareMathOperator{\Hom}{\mathrm{Hom}}
\DeclareMathOperator{\Ext}{\mathrm{Ext}}
\DeclareMathOperator{\Rep}{{\mathsf{Rep}}}
\DeclareMathOperator{\Vect}{{\mathsf{Vect}}}
\DeclareMathOperator{\Ker}{{\mathsf{Ker}}}
\newcommand{\en}{\enspace }
\newcommand{\mto}{\mapsto}

\newcommand{\inv}{^{-1}}
\newcommand{\vi}{${\en\mathsf {(i)}}\;$}
\newcommand{\vii}{${\;\mathsf {(ii)}}\;$}
\newcommand{\viii}{${\mathsf {(iii)}}\;$}
\newcommand{\iv}{${\sf {(iv)}}\;$}

\renewcommand{\dis}{\displaystyle}

\renewcommand{\mto}{\mapsto}
\newcommand{\into}{{}^{\,}\hookrightarrow^{\,}}
\newcommand{\too}{\,\longrightarrow\,}
\newcommand{\onto}{\twoheadrightarrow}
\newcommand{\cd}{\!\cdot\! }
\newcommand{\sym}{\operatorname{S}\!}
\newcommand{\U}{\operatorname{U}\!}

\def\C{\mathbb C}

\def\Z{\mathbb Z }
\def\OO{\gr }

\def\m{{\mathfrak{m}}}
\def\nf{\mathfrak{n}}

\def\tf{\mathfrak{t}}

\def\pb{$\bullet\quad$\parbox[t]{140mm}}

\def\wg{{\widetilde{\mathfrak{g}}}}
\def\wgc{{\widetilde{\g}^*_\chi}}

\def\B{{\mathcal B}}
\newcommand{\iso}{{\,\stackrel{_\sim}{\to}\,}}
\def\ccirc{{{}_{^{\,^\circ}}}}

\def\BG{{\mathbb G}}
\def\res{{\mathsf{res}}}
\def\WW{{\scr S}}

\def\L{{\cal L}}
\def\MM{{\mathcal M}}
\def\Om{\Omega}
\renewcommand{\o}{\otimes}

\def\hw{{\h^*\!/W}}
\def\pgr{{\scr P}(\gr)}

\def\<{\langle\kern-.08cm\langle}
\def\>{\rangle\kern-.08cm\rangle}

\def\XX{\overline{X}}
\def\YY{\overline{Y}}

\def\b{{\mathfrak b}}
\def\la{\lambda}

\def\fk{\mathfrak{k}}
\def\g{\mathfrak{g}}
\def\h{\mathfrak{h}}
\def\m{\mathfrak{m}}

\def\sat{{\mathbf S}}

\def\e{{e}}
\def\f{{f}}

\def\IH{{I\!H}}

\def\cg{\C[\g]}
\def\BT{{\mathbb T}}

\def\gde{\g^\e}

\newcommand{\ab}{{\hskip 8mm}}

\newcommand{\rs}{{\operatorname{reg}}}
\newcommand{\gr}{{\mathsf {Gr}}}
\newcommand{\gd}{{\mathfrak{g}}}
\newcommand{\hd}{{\mathfrak{h}}}
\newcommand{\Gv}{{{}^L\!G}}
\newcommand{\Tv}{{{}^L\!T}}

\newcommand{\sll}{\mathfrak{s}\mathfrak{l}_2}

\newcommand{\sset}{{\subset}}

\newcommand{\oo}{{\mathcal O}}
\begin{document}

\date{}

\title{\Large{\textbf{Variations on themes of Kostant}}
\author{\large
Victor Ginzburg
}}
\maketitle
\centerline{\em To Bert Kostant with admiration}
\vskip 25pt

\begin{abstract}
Let  $\g$ be a complex semisimple Lie algebra,
and let $\Gv$ be a complex  semisimple group with trivial center
whose root system is dual to that of $\g$.
We establish a graded algebra isomorphism   $H^\hdot(X_\lambda, \C)\cong
\sym \g^e/I_\lambda$, where $X_\lambda$ is an arbitrary spherical
 Schubert variety in the loop Grassmannian for $\Gv$,
and $I_\lambda$ is an appropriate ideal in  the symmetric algebra
of $\g^e$, the
 centralizer  of a principal nilpotent in $\g$.

We also discuss a `topological' proof of Kostant's  result 
on the structure of $\C[\g]$.
\end{abstract}


\section{Cohomology of `spherical' Schubert varieties}  \label{proof1}
In this paper, we discuss a few geometric results which
were, to a great extent,  inspired by three fundamental
papers of Bertram Kostant \cite{Ko1}-\cite{Ko3}.

\subsection{}\label{notat}
Let $\g$ be a complex semisimple Lie algebra and write
$\g^x\sset\g$ for the  centralizer of an element $x\in\g$.  We fix a 
principal $\sll$-triple $\langle h,\e,\f \rangle \subset \gd.$
Thus, $\h:=\g^h$  is a Cartan subalgebra of $\g$.
 The element $\e \in \gd$ is a principal
nilpotent and
we make the choice of positive roots of $(\g,\h)$ so that
$e$ is contained in the span of {\em simple} root vectors.

We write  $\hd_{_\Z}^*\sset\h^*$ for   the weight lattice of $(\g,\h)$.
Given an  finite dimensional $\g$-module $V$ and a  weight $\mu \in \hd_{_\Z}^*$,
let $V(\mu)\sset V$ denote the corresponding $\mu$-weight space,
let $\Specm V\sset\hd_{_\Z}^*$ be the set formed
by the weights which occur in $V$ with nonzero multiplicity,
ie. such that $V(\mu)\neq 0$.
We ignore weight multiplicities and let $I(V)\sset \C[\h^*]$
denote the ideal of polynomials vanishing at the  set  $\Specm
V_\lambda$,
viewed as a finite {\em reduced} subscheme in $\h^*$.
 Thus,  the coordinate
ring $\C[\Specm V_\lambda]=\C[\h^*]/I(V)$ is  a finite dimensional algebra.

Write $\U\fk$, resp. $\sym\fk,$ for the universal enveloping, resp. symmetric, algebra
of a Lie algebra $\fk$. A filtration on $\fk$ gives rise to a filtration
on $\U\fk$, resp. on $\sym\fk,$ not to be confused with the standard
increasing
filtration on an enveloping algebra.

Following B. Kostant and R. Brylinski,  on $\h,$ one introduces an increasing filtration
$F_\idot\h,$  where $F_k\h:=\{x\in\h\mid
\ad^{k+1}e(x)=0\}, \ k=0,1,\ldots$. The  induced filtration 
gives $\U\h$ the structure of a  nonnegatively filtered algebra.
Clearly, one has $\U\h=\sym\h=\C[\h^*]$.
In particular, we may (and will) view $\C[\Specm V_\lambda]$
as a quotient of the algebra $\U\h$. In this way, the  algebra $\C[\Specm
V_\lambda]$
acquires an increasing filtration that descends from the Kostant-Brylinski
filtration on $\U\h$. For the associated graded
algebras,
we have $\ggr^F\C[\Specm V_\lambda]=\ggr^F\U\h/\ggr^F I(V).$

\subsection{}\label{Ldual}
Write $\Gv$ for 
 a split connected
semisimple group (called {\em Langlands dual group}) that has  trivial center,
i.e., is of adjoint type, and that has a maximal torus, $\Tv\sset \Gv,$
such that, we have $\Lie\Tv=\h^*$ and the root system of $(\Gv,\Tv)$ is dual to that of $\g$.
Let $W$ denote the Weyl group, the same  both  for   $(\g,\h)$ and for
 $(\Gv,\Tv)$.

 Let $\C(\!(z)\!)$ be the field of formal
Laurent series and $\C[[z]] \sset\C(\!(z)\!)$ its ring of integers,
 the ring of
formal power series.
Let $\Gv(\!(z)\!)$, resp. $\Gv[[z]]$,
denote the group of  $\C(\!(z)\!)$-rational, resp. $\C[[z]]$-rational,
points of $\Gv$. 
The coset space $\gr := \Gv(\!(z)\!)/\Gv[[z]]$ is called
the {\em loop Grassmannian}.

The space $\gr$
has a natural structure of an ind-scheme,
a direct limit of an ascending chain
of projective varieties of increasing dimension, see
 e.g. \cite[\S 1.2]{G1} or \cite{Lu} for an elementary,
and \cite{BD} for a much more elaborate treatment.
There are canonical bijections
\beq{lattice}
\gr^{\,\Tv}=\Hom_\text{alg. group}(\C^\times,\Tv)=\hd_{_\Z}^*\en(\sset\h^*).
\eeq

Here, the space on the left denotes the fixed point set
of the left $\Tv$-action on $\gr$, and the $\Hom$-space in the middle denotes the coweight lattice of
$\Tv$.
Let $\OO_\lambda$ denote  the 
$\Gv[[z]]$-orbit of the   $\Tv$-fixed point
corresponding to such a coweight $\lambda$.
In this way, one obtains a bijection between the set of
 $\Gv[[z]]$-orbits in $\gr$  and
the set $\hd_{_\Z}^*/W$, of the Weyl group orbits in $\hd_{_\Z}^*$.
 The closure, $\overline{\OO}_\lambda\sset \gr$, to be referred to as 
`spherical Schubert variety',  is known
to be a finite dimensional projective variety,  singular in general.

\subsection{} \label{g^e}
For any Lie algebra $\fk$,
a $\fk$-module $V$, and a subset $S\sset V$, resp. $J\sset \U\fk,$ we put
\beq{eqnot}
\ann[\fk;\, S]:=\{ u\in \U\fk\mid u(s)=0\en \forall s\in S\},
\en\;\text{resp.}\en\;
V^J:=\{v\in V\mid u(v)=0\en \forall u\in J\}.
\eeq

Given an anti-dominant weight $\la \in\hd_{_\Z}^*$,
let $V_\la$ be an irreducible finite dimensional $\g$-representation 
with  lowest weight $\la$, and let $v_\la\in V_\la$ be  a lowest weight
vector. 

The centralizer of $e$ is  an $\ad h$-stable
abelian Lie subalgebra  $\gde\sset\gd$.
The grading induced by the $\ad h$-action
 makes  $\U\gde$  a  nonnegatively  graded
commutative
algebra. 
It is clear that 
 $\ann[\g^e;\, v_\la]$, the annihilator
of the element $v_\la$, is a graded ideal in
 $\U\gde$. Thus, the quotient $\U\gde/\ann[\g^e;\, v_\la]$ is a  graded finite
dimensional commutative algebra.

In this paper, we write $H^\hdot(-)$
for singular cohomology with $\C$-coefficients.
One of our main results reads

\begin{theorem} \label{peterson} 
The $\U\g^e$-action on $V_\la$ factors through the quotient
$\U\gde/\ann[\g^e;\, v_\la]$, and there are natural
graded algebra isomorphisms
\beq{2iso}H^\hdot(\overline{\OO}_\lambda) \simeq \U\gde/\ann[\g^e;\, v_\la]\simeq
\ggr^F\C[\Specm V_\lambda].
\eeq
\end{theorem}

\begin{remarks}\label{remarks} \vi
The  isomorphism on the left of \eqref{2iso}
 was communicated to me, as a conjecture, by Dale 
Peterson in the Summer of 1997.  This   isomorphism is induced by
the restriction map $H^\hdot(\gr) \to H^\hdot(\overline{\OO}_\lambda,
\C)$, using a natural graded algebra isomorphism 
$\dis H^\hdot(\gr) 
\simeq
\U\gde,$ proved in [G1, Proposition 1.7.2].
\smallskip

\vii
The isomorphism  on the right  of \eqref{2iso} is induced by the
Kostant-Brylinski  graded algebra isomorphism,
$\ggr^F \U\h\cong \U\gde,$ cf. ~\cite{Bry}.
In other
words,  the second isomorphism in \eqref{2iso} amounts to an equality
$\ggr^F I(V_\la)=\ann[\g^e;\, v_\la]$, of graded ideals in $\U\g^e$.
\smallskip

\viii The composite  isomorphism in  \eqref{2iso} 
 provides a
description of the cohomology of $\overline{\OO}_\lambda$
similar to the
description of the cohomology of  the  flag variety 
$G/B$, due to  Kostant. Such a description was furher extended to the case
of arbitrary Schubert varieties in $G/B$ by  Akyildiz and Carrell \cite{AC}.
It would be interesting to obtain a direct proof of our
isomorphisms using the techniques of  \cite{AC},
at least in the simply laced case.

\iv There is an
analogue of Theorem \ref{peterson} for an {\em arbitrary},
not necessarily spherical Schubert variety $X_\la\sset\gr$.
  Specifically,
in \cite{GK}, \S 4.3, the authors have associated
with any, not necessarily anti-dominant, weight $\la\in\h^*_{\Z}$,
a certain finite dimensional graded $\U\g^e$-module $F_\la$.
The module $F_\la$ was defined
in terms of quantum group cohomology, and it was
conjectured that the same module has an alternate
geometric interpretation in terms of
intersection cohomology of the Schubert variety $X_\la$.
That conjecture was later proved in \cite{ABG}, Theorem 10.2.3.
Using these results, it is straightforward to
extend our proof of Theorem \ref{peterson}
to obtain an algebra isomorphism of the form
$H^\hdot(X_\la)\cong \U\g^e/\ann[\g^e;\, F_\la]$.
\end{remarks}

\subsection{}
From Theorem \ref{peterson},
comparing dimensions of $\C[\Specm
V_\lambda]$ and of $V_\la$,
 we deduce 
\begin{corollary}\label{mult1} One has $V_\la=\U\gde\cdot v_\la$ if
and only if any weight occurs in $V_\la$ with multiplicity
$\leq 1$.\qed
\end{corollary}

We remark that there are very few representations, besides the minuscule
ones,
with all
weights of multiplicity at most one. For simple groups, these are: $\sym^n(\C^n)$
for $G=SL(n)$; the standard $2n+1$-dimensional  representation for 
 $G=SO(2n+1)$, and a
7-dimensional  representation for a group $G$ of type $G_2$.

Let $Q(\g,\h)\sset \h^*_{_\Z}$ be the root lattice.
In the special case where  $\la\in Q(\g,\h),$
Corollary \ref{mult1} is equivalent to 
the following unpublished result of B. Kostant.

\begin{proposition}\label{KKK} Let  $\lambda\in Q(\g,\h)$ be a anti-dominant
weight. 
Then,  $V_\la=\U\gde\cdot v_\la$ 
holds if and only if $\dim V_\la(0)=1.$ 
\end{proposition}

\begin{proof}\footnote{We are indebted to B. Kostant for kindly communicating to us
both the statement and the main idea of the short direct prooof of this
result reproduced below.}
Fix an arbitrary  anti-dominant weight  $\lambda$ contained in the
root lattice of $\gd$ and  an $h$-stable direct sum decomposition
$V_\la= \gde(V_\la)\oplus M$.
We claim that
\beq{UE}
V_\la=\U\gde\cd M \qquad\text{and}\qquad\dim M=\dim V_\la(0).
\eeq

To prove the first equation, assume $v\in V_\la$ is a nonzero vector that does not belong to
$\U\gde\cd M.$ We may assume without loss of generality that $v$ is
an $h$-eigenvector with an integral eigenvalue $c$. Furthermore, let $v$
be such that
$c$ is the minimal possible among those  arising in this way for
various $h$-eigenvectors $v$ as above.

Now, since $V_\la= \gde(V_\la)\oplus M$, we have $v=\sum_j a_j\cdot v_j +m$,
for some  $\ad h$-weight vectors $a_j\in \gde$  and some 
$h$-eigenvectors $m\in M,\,v_j\in V_\la,$ such that $v_j\not\in\U\gde\cd
M$
for at least one $j$.
Write $c_j\in\Z$ for
the $h$-eigenvalue of the vector $v_j$.
The weights of the $\ad
h$-action
on $\gde$ being positive integers (equal to twice the exponents of $\gde$), we deduce that
$c_j<c$. Hence, the minimality of $c$ forces
 $v_j=0$ for all $j$. Thus, $v=m\in M$.
This is a contradiction, and the first equation in  \eqref{UE} follows.

To prove the second equation,
let $V_\la^*$ denote the contragredient representation.
By a result of Kostant, cf. Theorem \ref{key}(ii) below,
one has $\dim (V_\la^*)^{\gde}=\dim (V_\la^*)(0)=\dim V_\la(0).$
Furthermore, the annihilator of  the
subspace $(V_\la^*)^{\gde}$ with respect to the canonical perfect
pairing
$V_\la\times V_\la^*\to \C$ equals $\gde(V_\la)\sset V_\la$.
Thus, we deduce $\dis
\dim \gde(V_\la)+\dim V_\la(0)=\dim V_\lambda,$
and the  second equation in  \eqref{UE} follows.

Let finally  $\dim V_\la(0)=1$. Observe that  $v_\la\not\in\gde(V_\la)$
so, by the  dimension formula in
\eqref{UE},
the space  $M=\C\cdot v_\la$ provides, in this case, a direct complement to $\g^e(V_\la)$.
Then, the first equation  in  \eqref{UE} 
yields  $V_\la=\U\gde\cdot v_\la$.
This proves one implication; the proof of the opposite implication is
similar
and is left to the reader.
\end{proof}

We note that  the dimension equality in \eqref{UE} clearly holds in the setting 
of Corollory \ref{mult1}  as well.
However, the proof of Proposition \ref{KKK} does not go through since
the equation $\dim V_\la^e=\dim V_\la(0)$ fails unless
$\la\in Q(\g,\h)$.

\subsection{} We will use simplified notation
$\sym\h^W$ for the algebra $(\sym\h)^W=\C[\h^*]^W
\cong\C[\Lie\Tv]^W.$
Let $\m_o:=\sym\h^W_+$ denote the augmentation ideal of  $\sym\h^W$.
Given  a weight $\la$,  write $W_\la\sset W$ for the
isotropy group of  $\la$, and let  $P_\la\sset\Gv$ be the corresponding parabolic
subgroup.
 Thus, one has graded algebra imbeddings
$\sym\h^W\into  \sym\h^{W_\la}\into \sym\h.$

Further, view $\la$ as a $\Tv$-fixed point in the loop Grassmannian.
The  corresponding $\Gv$-orbit, 
$\Gv\cdot\la\cong \Gv/P_\la$, is  a partial flag manifold.
In this way, we get a chain of imbeddings
\beq{u2v2} 
\xymatrix{
\Gv/P_\la=\Gv\cd\la\en\ar@{^{(}->}[rr]^<>(0.5){\imath_\la}&&
 \en\Gv[[z]]\cd\la=\OO_\la\en\ar@{^{(}->}[rr]^<>(0.5){\jmath_\la}&&
\en\overline{\OO}_\la,
}
\eeq
where the map $\imath_\la$ is a {\em homotopy equivalence}
and the  map $\jmath_\la$ is an open imbedding.
Therefore, the composite map in \eqref{u2v2} 
induces, via  Theorem \ref{peterson},  graded algebra 
morphisms
\beq{uuvv}
\xymatrix{
\U\gde/\ann[\g^e;\, v_\la]=H^\hdot(\overline{\OO}_\la)
\ar[rr]^<>(0.5){(\jmath_\la\ccirc \imath_\la)^*}&&
 H^\hdot(\Gv/P_\la)=
\sym\h^{W_\la}/\m_o\cd\sym\h^{W_\la},
}
\eeq
where the last equality 
is  the  Borel isomorphism.

\begin{proposition} \label{peterson2}
For any anti-dominant weight $\la$,
 the map \eqref{uuvv} is surjective. Moreover,  it may be factored naturally as a composite 
\beq{twomaps}
\xymatrix{
\U\gde/\ann[\g^e;\, v_\la]\
\ar@{^{(}->}[rr]^<>(0.5){u\mapsto u(v_\la)}&&
\ V_\la\ \ar@{->>}[rr]^<>(0.5){f_o} &&\
\sym\h^{W_\la}/\m_o\cd\sym\h^{W_\la}.
}
\eeq 
\end{proposition}

Friedman  and Morgan \cite{FM}, have introduced
a certain linear map
$f: V_\la\to \sym\h^{W_\la},$ 
 and  conjectured
that the composite
$\text{proj}\ccirc f:\
V_\la\to\sym\h^{W_\la}\onto\sym\h^{W_\la}/\m_o\cdot\sym\h^{W_\la},$
is a surjective map. It is not difficult to verify, going through the constructions
of the maps in question, that one actually has $f_o=\text{proj}\ccirc f.$
Thus, our proof of Proposition \ref{peterson2}, to be given in
 \S\ref{proof2},  yields at the same time
a proof of (a  strengthening of) the Friedman-Morgan conjecture.

\subsection{} We recall  that the orbit $\OO_\mu$ is closed in $\gr$ if and only if
the weight $\mu$ is either {\em minuscule} or zero. 
In such a case, each of the imbeddings
in \eqref{u2v2} is  an equality and the restriction map
$(\imath_\mu\ccirc\jmath_\mu)^*,$ in \eqref{uuvv}, becomes an isomorphism.

\begin{theorem}\label{key} Let  $\mu$ be either a minuscule
anti-dominant weight
or zero. Then,\vskip 2pt

\vi Each of the two maps
in \eqref{twomaps} is a bijection, and the composite
yields
a graded algebra isomorphism
$\dis \U\gde/\ann[\g^e;\, v_\mu]\cong\sym\h^{W_\mu}/\m_o\cd\sym\h^{W_\mu}.$
\vskip 2pt

\vii For any finite dimensional $\g$-module $V$ such that
$\Specm V\sset\,
\mu+Q(\g,\h),$  we have 
$$\dim V^{\ann[\g^e;\, v_\mu]}/|W|
=\dim V(\mu)/|W_\mu|. $$
\end{theorem}

The bijectivity of the first map
in \eqref{twomaps} for a  minuscule weight $\mu$
is immediate from Corollary \ref{mult1}, see also
 \cite{G1}, Proposition 1.9,  Proposition 1.8.1.
The bijectivity of  the second map $f_o$ will be proved in \S\ref{pf_key}.
It 
has also been established  earlier by Friedman and Morgan~\cite{FM},
in a totally different way.
\begin{remarks} (1)
For $\mu=0$, the statement of Theorem \ref{key}(ii) is  a classic
result of Kostant \cite{Ko1}, cf. also Theorem \ref{orig} below. Note that
this special case of Theorem \ref{key}(ii) immediately implies, by a
standard
semicontinuity argument, another important result of Kostant
saying that {\em the function $x\mto\dim V^{\g^x}$
is constant on the set of regular elements of $\g$, whenever}
$\Specm V\sset Q(\g,\h)$.

(2)  An alternative much more elementary proof of
the graded space isomorphism from  Theorem \ref{key}(i) 
is due to B. Gross  \cite{Gr}, Proposition 6.4. The technique
of {\em loc. cit.} was insufficient, however, for establishing
the isomorphism of {\em algebras}.
\end{remarks}

\vskip 4mm
\noindent
{\sc Acknowledgements.}   {\footnotesize I am grateful to Dale Peterson and 
Bert Kostant for interesting discussions, and also
to  Mark Reeder for several very helpful comments on an earlier draft of the paper,
 and to Braham Broer  for bringing reference
 \cite{Bro} to my attention.}

\section{A geometric result}  \label{main}

Our proof of  Theorem \ref{peterson} is based on a simple general
 result concerning intersection cohomology  which we are going to state
 and prove below. Write $\otimes=\otimes_{_\C}.$ 

\subsection{}
Let $Y$ be a (not necessarily smooth) complex  algebraic variety,
and $D^b(Y)$ the bounded derived category of constructible complexes
of $\C$-sheaves on $Y$.
Let $\C_Y\in D^b(Y)$ be the constant sheaf,
resp.  $IC(Y)\in D^b(Y)$ be  the
intersection cohomology complex on $Y$ i.e.,  the 
`intermediate extension' of the constant sheaf
on a Zariski open smooth subset $Y_0\sset Y$.
In this paper, we will use the `naive' normalization in which the restriction
of $IC(Y)$ to $Y_0$ is a constant sheaf concentrated in degree 0 (not in degree
$-\ddim Y $ as in \cite{BBD}). Thus non-zero cohomology sheaves, ${\cal H}^iIC(Y)$,
may occur only in degrees $0\leq i\leq \ddim Y$.
We
write  $\IH^\hdot(Y)=H^\hdot(Y,
IC(Y))$,
for the intersection cohomology groups of $Y$ with
complex coefficients.

We assume $Y$ is irreducible. Then, in degree zero, we have
$\dis\C=H^0(Y)=\IH^0(Y)=\Hom_{D^b(Y)}(\C_Y,IC(Y))$.
Therefore, there exists a nonzero morphism
$\kappa: \C_Y\to IC(Y)$, unique up to a constant factor.
This gives the induced map of (hyper-) cohomology,
\begin{equation}\label{kappa}
\kappa_*:\; H^\hdot(Y) = H^\hdot(Y,\,\C_{Y})\, 
\longrightarrow\, H^\hdot(Y,\,IC(Y))=\IH^\hdot(Y)\,.
\end{equation}

\subsection{}
Let $Y$ be an irreducible 
projective variety with an algebraic
stratification $Y = \sqcup_{s\in \WW}\, Y_s.$
The strata $Y_s$
are assumed to be smooth locally closed subvarieties labelled by the elements
of a finite set $\WW$. 

We will use the standard yoga of {\em weights} introduced by
Deligne \cite{De}, cf. also \cite{BBD}.
 Recall that a mixed complex $A$ on $Y$ 
is said to be {\em pointwise pure} if $i^*_yA$ is pure
for any $y\in Y$, where
$i_y: \{y\}\into Y$ denotes the
imbedding.

Our
geometric result about intersection cohomology reads

\begin{proposition}  \label{prop1}
Assume the following two conditions hold:

\pb{The complex $IC(Y)$ is pointwise pure;}

\pb{Each stratum $Y_s$ is isomorphic to an affine (e.g. vector) bundle over
a smooth con-nected, and simply-connected projective variety.}
\vskip 3pt

\noindent
Then,  the map
$\kappa_*$ in \eqref{kappa} is injective,  equivalently,
the dual map
$\IH^\hdot(Y) \rightarrow H_\idot(Y)$, to homology,
is surjective.
\end{proposition}

\begin{proof} For each $s\in \WW$,  write $\YY_s$ for the closure of $Y_s$
and let $j_s:
Y_s \hookrightarrow Y$, resp.
$j'_s:
Y_s \hookrightarrow \YY_s,$ denote the natural inclusion.
Then,
for any $A\in D^b(Y),$ there is a standard convergent spectral sequence 
of the form
\beq{E2} E^{p,q}_2(A)= \bigoplus_{\dim Y_s = -p}\, H^{p+q}\big(\YY_s,\ (j'_s)_!j^*_sA\big)\en
\Rightarrow\en E^{p,q}_\infty(A)=\ggr_pH^{p+q}(Y, A)\,,
\eeq
where in the last equality on the right we have used
that  $H^\hdot_c(Y, A)=H^\hdot(Y, A)$ since $Y$ is compact.

First, let $A=\C_Y$. Then, $j^*_sA=\C_{Y_s}$, hence, we have
$H^n\big(\YY_s,\ (j'_s)_!j^*_sA\big)= H^n_c(Y_s),$ the
cohomology of $Y_s$ with compact support. By the second assumption of
the proposition, 
the cohomology $H^\hdot_c(Y_s)$ is pure. It follows that, for  $A=\C_Y$, the corresponding spectral
sequence \eqref{E2}
degenerates at the $E_2$-term, by a standard argument due to Deligne. Thus, we obtain
$E_2(\C_Y)=E_\infty(\C_Y)=\ggr H^\hdot(Y).$

Next, we let $A=IC(Y)$. Then, each cohomology sheaf ${\mathcal H}^pj^*_sA$
is a locally constant, hence constant, 
sheaf on $Y_s,$ for any $s\in \WW$. Furthermore, the pointwise purity
assumption on $IC(Y)$ insures that each of the local systems  ${\mathcal
H}^pj^*_sIC(Y)$
is pure. Hence, the spectral
sequence \eqref{E2}  for  $A=IC(Y)$ degenerates 
at  the $E_2$-term again.  Thus, we obtain
$E_2(IC(Y))=E_\infty(IC(Y))=\ggr \IH^\hdot(Y).$

Now, the canonical sheaf
 morphism $\kappa: \C_Y\to IC(Y)$ induces a morphism
of spectral sequences associated with $\C_Y$ and with  $IC(Y)$,
respectively. For the corresponding $E_2$-terms, the  resulting map 
$E^{p,q}_2(\C_Y)\to 
E_2^{p,q}(IC(Y))$
splits into a direct sum of the form
\beq{EE}\bigoplus_{\dim Y_s = -p}\,\left[H^{p+q}_c(Y_s,\C_{Y_s})
\  \stackrel{\kappa_s}\too \  H^{p+q}_c\big(Y_s,\
{\mathcal
H}^\hdot j^*_sIC(Y)\big)\right].
\eeq

Here,  $\kappa_s$ is a map induced by the  morphism of sheaves
$\kappa|_{Y_s}:\ \C_{Y_s}\to {\mathcal
H}^0 j^*_sIC(Y)$, the
restriction of the canonical 
 morphism $\kappa: \C_Y\to {\mathcal H}^0IC(Y)$ to $Y_s$.
Since $\kappa$ is clearly injective,
the restriction of  $\kappa$ to $Y_s$ is injective as well.
It follows that each of the maps  $\kappa_s$, hence
the map \eqref{EE} itself, is injective.

Since both spectral sequences degenerate at
the corresponding $E_2$-terms, we conclude
that we have proved injectivity of the map
$$ E_2(\C_Y)=\ggr H^\hdot(Y,\C_Y)\too
E_2(IC(Y))=\ggr \IH^\hdot(Y).$$
The latter map is nothing but the map $\ggr\kappa_*$, the associated graded map for the
morphism \eqref{kappa}.
Thus, the injectivity of  $\ggr\kappa_*$ implies
injectivity of the map $\kappa_*$ itself.
\end{proof}

\subsection{Equivariant setting.}\label{equiv_sec}
Applications of Proposition \ref{prop1} often involve varieties with a
group
action.
Let  $\BG$ be  a complex connected reductive group and $pt:=$ point.

Given  a $\BG$-variety $Y$,
there is a standard functor
 $H^\hdot_\BG(Y,-)$, of  $\BG$-equivariant 
cohomology, see \cite{BL1}.
This functor assigns to a $\BG$-equivariant constructible complex $\L$, on
 $Y$, a graded $H^\hdot_\BG(pt)$-module  $H^\hdot_\BG(Y,\L)$, the  total $\BG$-equivariant 
(hyper-) cohomology group of $\L$.
In the special case of the constant sheaf $\C_Y$, we have
$H^\hdot_\BG(Y,\C_Y)= H^\hdot_\BG(Y)$,  the
 $\BG$-equivariant cohomology of the space $Y$. This is a graded
$H^\hdot_\BG(pt)$-algebra; moreover, it is   a finite rank free graded
$H^\hdot_\BG(pt)$-module provided $Y$ has a $\BG$-stable stratification
$Y=\sqcup_{s\in\WW}Y_s$, where each stratum
$Y_s$ is  (equivariantly) isomorphic to a $\BG$-equivariant
affine bundle on a smooth projective $\BG$-variety.

One can define $\BG$-equivariant intersection cohomology $\IH^\hdot_\BG(Y):=
H^\hdot_\BG(Y,\ IC(Y)),$ which is also known to be  a finite rank free graded
$H^\hdot_\BG(pt)$-module, cf. \cite{GKM}. Using this, it is
straightforward
to extend the proof of Proposisition \ref{prop1} to an
equivariant setting. Thus, one obtains the following result

\begin{corollary}\label{cor1} Let $Y=\bigsqcup_{s\in\WW} Y_s$ be a
stratification of a projective $\BG$-variety $Y$ such that the strata
$Y_s$ are $\BG$-stable and satisfy the conditions of Proposition
\ref{prop1}. Then, the natural  map
$H^\hdot_\BG(Y) \rightarrow \IH^\hdot_\BG(Y)$ is a {\em split}
imbedding of finite rank free $H^\hdot_\BG(pt)$-modules.\qed
\end{corollary}

The space $H^\hdot_\BG(Y,\L)$
may be viewed as a graded $\C[\Lie\BG]^{\Ad\BG}$-module,
via the standard graded
algebra isomorphism $H^\hdot_\BG(pt)=\C[\Lie\BG]^{\Ad\BG}.$
For any {\em semisimple} element $t\in \Lie\BG$,
let $\m_t\sset \C[\Lie\BG]^{\Ad\BG}$ denote the maximal
ideal corresponding to  $\Ad\BG(t)$, a closed conjugacy class in
$\Lie\BG$.
The quotient
$H_t(Y,\L):=H^\hdot_\BG(Y,\L)/\m_t\cdot H^\hdot_\BG(Y,\L)$
will be referred to as a {\em specialized} equivariant cohomology
group of $\L$. The grading on $H^\hdot_\BG(Y,\L)$ gives rise
to a canonical ascending {\em filtration} $E_\idot H_t(Y,\L),$ on $H_t(Y,\L)$.

The augmentation ideal $\m_o\sset \C[\Lie\BG]^{\Ad\BG}$
 is a {\em
graded}
ideal. Hence, for  $t=o\in \Lie\BG,$
 the origin, the grading on $H^\hdot_\BG(Y,\L)$ descends to
a well defined  grading on $H_o(Y,\L)$. 
In the case where
$\L$ is, in addition, a semisimple  $\BG$-equivariant perverse sheaf of geometric
origin, in the sense of \cite{BBD},  one has
a natural graded space isomorphism with the non-equivariant cohomology
of $\L$, cf. \cite{G1},
\beq{fixo}
\ggr^E_\idot H_t(Y,\L)\cong H^\hdot_o(Y,\L)\cong H^\hdot(Y,\L),
\qquad t\in \Lie\BG.
\eeq

A similar result also holds for the constant sheaf $\L=\C_Y$, provided
the spectral sequence for the equivariant cohomology of $Y$ collapses, see eg.
\cite{GKM}. This is the case, for instance, if the odd cohomology of
$Y$ vanishes. The latter condition  
is known to hold  for any spherical Schubert variety
$Y=\overline{\OO}_\la$.

Next, let $\BT$ be a complex torus. Fix a {\em smooth} complex  projective 
$\BT$-variety $X$  and a subgroup  $\C^\times\sset\BT$.
Assume that the $\C^\times$-action $t: x\mto t\cdot x$,   on $X$,
has a {\em finite}  fixed point set $\WW=X^{\C^\times}$.
For each fixed point $s\in \WW$, define the corresponding attracting set,
$$ X_s = \{ x\in X\enspace\mid\enspace
\lim_{t \to 0}\, t\cdot x = s\},\qquad
 t\in \C^\times.$$

One has the Bialynicki-Birula cell decomposition 
$X = \sqcup_{s\in \WW}\, X_s$. 
Let $\XX_s$ be the closure of  $X_s$. For any $s\in \WW,$ the cell $X_s$
is isomorphic to a $\C$-vector space, see \cite{BB}, moreover, the complex
$IC(\XX_s)$ is known to be pointwise pure,
cf. eg. \cite{KL} or \cite{G2}.

Thus, Corollary \ref{cor1} yields 

\begin{corollary}  \label{thm1}
Assume that 
the  Bialynicki-Birula decomposition $X=\sqcup_{s\in \WW}\, X_s $ is a
stratification of
$X$, and let $s\in \WW$.
Then, the map
$\kappa_*: H^\hdot(\XX_s) \rightarrow \IH^\hdot(\XX_s),$ 
resp.  the map $\kappa_t: H_t(\XX_s) \rightarrow \IH_t(\XX_s),$ is
injective
for any  $t\in \Lie\BT$.\qed
\end{corollary}

\section{Proofs: application of the Satake equivalence}
\subsection{} Write
 $E^*$ for the dual of a vector space $E$, and recall from
 \S\ref{notat} the $\sll$-triple $e,h,f\in\g.
$ Let $\rho\in\h^*$ be the half-sum
of positive roots.
We will identify elements of $\h^*$ with
linear functions on $\g$ via the standard
imbedding $\h^*\into\g^*$.

Given $\la\in \g^*$, write  $\g^\la\sset\g$ for the isotropy Lie algebra
of the  $\ad^*\g$-action on $\g^*$ (the identification
$\g\simeq \g^*$ provided by an invariant form,
makes $\g^\la$  the centralizer of $\la$).
We set $\eta:=\ad^*e(\rho).$ It is easy to check that
one has $\g^\eta=\g^e$, resp. $\g^{\eta+\rho}=\g^{e+h}$.
According to Kostant  \cite{Ko3}, for any $\mu\in\h^*$, 
the Lie algebra $\g^{\eta+\mu}$ is abelian and  one has
$\dim\g^{\eta+\mu}=\rk\g.$

Let  $\Rep\g$ be the  tensor category   of finite dimensional
representations of the Lie algebra $\g$.
Let  $h$ be the semisimple element from the $\sll$-triple, and let  $V\in \Rep\g$.
The eigen-space decomposition with respect to the
$h$-action in $V$
gives  a grading $V_\idot=\bigoplus_{m\in\Z}\ V_m$,
to be referred to as the {\em Brylinski grading}.
We also define an increasing 
filtration $F_\idot V$, called the {\em Brylinski filtration},
defined by $F_kV:=\bigoplus_{m\leq k} V_m$.

In the special case where $V=\g$, the adjoint representation,
the Brylinski filtration restricts to a filtration on any Lie subalgebra
$\fk\sset\g$. This makes the enveloping algebra $\U\fk$ a filtered
algebra with respect to the induced  filtration (not to be
confused with the standard filtration on an enveloping algebra).

Write $\Vect$
(resp. $\Vect_\idot$), for the tensor category of finite dimensional
{\em filtered}  resp. graded, vector
spaces.
Assigning to a $\g$-representation the underlying vector space
equipped with the  Brylinski filtration, resp.
Brylinski grading,
yields a fiber functor $F: \Rep\g\to\Vect,$
resp. $\Rep\g\to\Vect_\idot.$

\subsection{}
We have the Langlands dual data $(\Gv,\Tv)$, see \S\ref{Ldual},
and we put $\tf=\Lie\Tv.$
Thus, there is a diagram of canonical morphisms
of affine algebraic varieties
\beq{transfer}
\xymatrix{
\eta+\h^*\ar[rrr]^<>(0.5){\eta+\mu\,\mto\, W(\mu)} &&&
 \;\hw{\;\simeq\;}\tf/W{\;\simeq\;} \Specm\big(\C[\Lie\Gv]^{\Ad\Gv}\big),
}
\eeq
where  the last isomorphism is the Chevalley isomorphism.

We may  view
 $\Gv$ as a subgroup  of $\Gv[[z]]$ formed by 
constant loops. The loop Grassmannian thus becomes a
 topological $\Gv$-space with respect to the left $\Gv$-action.
Therefore, associated with any $\Gv$-equivariant complex $\L$,
on $\gr$, and an element $\mu\in\h^*=\tf,$ there is a  filtered vector space
$H_\mu(\gr,\L)$, a specialized $\Gv$-equivariant cohomology group, see \S\ref{equiv_sec}.
Equivalently, the same group may be obtained
 by  specializing $H^\hdot_\Tv(\gr,\L)\cong
\C[\tf]\bigotimes_{\C[\tf]^W}H^\hdot_\Gv(\gr,\L)$,
the $\Tv$-equivariant cohomology,
at the point $\mu\in\tf.$

In   [G1, Proposition 1.7.2], we have proved, cf.
 also \cite{BFM},

\begin{proposition}\label{HGGr}
For any $\mu\in\h^*$, there is a canonical  filtered algebra isomorphism
\beq{alg}
\varphi_\mu:\ H_\mu(\gr)\iso \U\g^{\eta+\mu}.
\eeq

For $\mu=o$ (the origin), this isomorphism respects the canonical gradings on both sides.
\end{proposition}

\begin{remark}\label{mu0} For $\mu=o,$ we have
$H^\hdot_o(\gr)=H^\hdot(\gr),$
by
  \eqref{fixo}, and \eqref{alg}  reduces
to the graded isomorphism mentioned in Remark \ref{remarks}(ii).
\end{remark}

\subsection{}\label{proof2}
Let $\pgr$ be the tensor category 
of  $\Gv[[z]]$-equivariant 
perverse sheaves on $\gr$ with compact support, and equipped with a monoidal
structure given by  convolution.
The category $\pgr$ is known to be semisimple, cf. eg. \cite{MV}.
Taking specialized equivariant cohomology 
gives a family fiber
functors parametrized by the elements $\mu\in\h^*,$
$$
H_\mu(\gr,-): \pgr\to \Vect\quad \big(\text{resp.}\en H^\hdot_o(\gr,-):
\pgr\to \Vect_\idot\big),
\quad \L \mto H_\mu(\gr,\ \L).
\qquad
$$ 

Using the isomorphism in  \eqref{fixo}, we will often identify specialized and ordinary cohomology
functors $H^\hdot_o(\gr,-)=H^\hdot(\gr,-).$

View $\rho$ as an element of $\tf$ and observe that the $\rho$-fixed
point set in $\gr$ is equal to the set \eqref{lattice}, of 
$\Tv$-fixed points. 
  For each $\mu\in \h_{_\Z}^*$,
let $i_\mu: \{\mu\}\into\gr$ denote the corresponding one
point imbedding. 
The Localization theorem in equivariant cohomology
yields
 the following {\bf fixed point decomposition}, see [G1, (3.6.1)], 

\begin{equation}\label{fp_decomposition}
H_\rho(\gr,\L)= \oplus_{\mu\in\h_{_\Z}^*}\en H_\rho(i_\mu^!\L),\qquad \forall \L\in \pgr\,.
\end{equation}

\begin{theorem}[Geometric Satake equivalence]\label{sat}
\vi There is a
monoidal equivalence  $\dis\sat: \pgr\iso \Rep\g\,$ and, for each $\mu\in\h^*$,  an isomorphism 
$\Phi_\mu:\ H_\mu(\gr,-)$ $\iso
F\ccirc \sat(-)$,  of monoidal functors, such
that
\vskip 3pt

\noindent
\pb{For any dominant weight $\la$, one has an isomorphism
$\ \dis \sat(IC(\overline{\OO}_\lambda))\cong V_\la.$}

\noindent
\pb{For any $\L\in \pgr$ and $u\in H_\mu(\gr)$, the following diagram
commutes:}

\beq{act_diag}
\xymatrix{
H_\mu(\gr,\L)\ar@{=}[d]_<>(0.5){\Phi_\mu}\ar[rrr]^<>(0.5){u:\ c\mto
u\cup c}&&&
H_\mu(\gr,\L)\ar@{=}[d]_<>(0.5){\Phi_\mu}\\
F\ccirc\sat(\L)\ar[rrr]^<>(0.5){F({\varphi_{\mu}(u)})}&&&
F\ccirc\sat(\L).
}
\eeq

\noindent
\pb{The  isomorphism
$\Phi_\mu$ sends the canonical filtration on the specialized equivariant
cohomology
 to the Brylinski filtration.}
\vskip 3pt

\noindent
\vii\parbox[t]{140mm}{For $\mu=\rho$, the  isomorphism
$\Phi_\rho$ sends the 
fixed point decomposition \eqref{fp_decomposition}
 to the weight decomposition
with respect to the Cartan subalgebra $\g^{\eta+\rho}$.}
\vskip 3pt

\noindent
\viii\parbox[t]{140mm}{For $\mu=o$, the natural cohomology grading on 
$H_o^\hdot(\gr, -)$ goes, under the isomorphism
$\Phi_o$, to the  Brylinski grading.\qed}
\end{theorem}

This theorem was first proved in \cite{G1},
(see Theorem 1.4.1 and
Theorem 1.7.6 in  \cite{G1}), using some earlier results
by Lusztig \cite{Lu}, and a crucial idea due to
Drinfeld. Exploiting some additional
unpublished ideas of 
Drinfeld  concerning the commutativity constraint in $\pgr$,
an alternate, more elegant proof of the monoidal equivalence
 of the theorem was  obtained later
by Mirkovi\'c and Vilonen \cite{MV}.

\subsection{} Recall the notation from \eqref{eqnot} and \eqref{kappa}.

\begin{lemma}\label{lem} For any $\mu\in\h^*,$ and an anti-dominant weight $\la\in\h_{_\Z}^*$,
 there is a filtered algebra isomorphism
\beq{filt}
H_\mu(\overline{\OO}_\la)\cong
\U\g^{\eta+\mu}/\ann[\g^{\eta+\mu};\, v_\la].
\eeq

Furthermore, in the special
case  $\mu=o$, the above isomorphism respects the gradings.
\end{lemma}
\begin{proof} We write $\ann[-;\, \kappa_*(1)]$ for an annihilator of the
element $\kappa_*(1)\in \IH_\mu(\overline{\OO}_\lambda).$

One has a natural  restriction map
$\res_\la: H_\mu(\gr)\to
H_\mu(\overline{\OO}_\la)$.
This map  is surjective since
the spaces $\overline{\OO}_\la$ and $\gr$ have compatible
cell decompositions by  Iwahori-orbits (which have  {\em even real}
dimension).
By 
Corollary \ref{thm1} applied to the variety
  $X=\overline{\OO}_\lambda$,
we have
$$
\res_\la\big(\ann[H_\mu(\gr);\, \kappa_*(1)]\big)\sset
\ann[H_\mu(\overline{\OO}_\lambda);\, \kappa_*(1)]=0.
$$

This yields an inclusion
$\ann[H_\mu(\gr);\, \kappa_*(1)]$
$\sset\Ker(\res_\la).$
Clearly, the $H_\mu(\gr)$-action on $\IH_\mu(\overline{\OO}_\la)$
factors through the restriction $H_\mu(\gr)\onto
H_\mu(\overline{\OO}_\la)$.
Hence, we must have $\ann[H_\mu(\gr);\, \kappa_*(1)]=\Ker(\res_\la).$
The map $\res_\la$ therefore induces an isomorphism
\beq{annann}
H_\mu(\gr)/\ann[H_\mu(\gr);\, \kappa_*(1)]\iso
H_\mu(\overline{\OO}_\lambda).
\eeq

On the other hand, 
the isomorphism
$\varphi_\mu: H_\mu(\gr)\iso \U\g^{\eta+\mu}$ in
\eqref{alg} provides, by  Corollary \ref{sat}, the following  bijections
\begin{align*}
&\ann[H_\mu(\gr);\, \kappa_*(1)]\iso\ann[\g^{\eta+\mu};\, v_\la],
\quad\text{resp}.\\
&\ann[H_\mu(\gr);\, \IH_\mu(\overline{\OO}_\la)]\iso
\ann[\g^{\eta+\mu};\, V_\la].
\end{align*}
The isomorphism of the lemma now follows  from \eqref{annann}.
\end{proof}

\begin{proof}[Proof of Theorem \ref{peterson}] The first isomorphism in
\eqref{2iso}
is the statement of Lemma \ref{lem} for $\mu=o$.
To prove the second isomorphism in \eqref{2iso}, we take $\mu=\rho$.
By Proposition \ref{HGGr}, we get
$H_\rho(\gr)=\U\g^{\eta+\rho}=\U\g^{e+h} {\;\cong\:}\U\g^h =\U\h,$
where the  third isomorphism is given by conjugation by
$u:=\exp e,$ a
unipotent  element. It is straightforward to verify
that such a conjugation sends the Brylinski filtration 
on $\g^{\eta+\rho}$ to the filtration on $\h$
that was defined in \S\ref{g^e}.
Further, it is clear from Theorem \ref{sat}(ii) 
that   the ideal
$\ann[U(g^{\eta+\mu});\, V_\la]\sset \U\g^{\eta+\rho}$
goes, via conjugation by $u$, to
 the ideal
$\ann[\h;\, V_\la]= I(V_\la)\sset \U\h$. 
Thus, we obtain 
filtered algebra isomorphisms
\beq{filtalg}
\xymatrix{
H_\rho(\overline{\OO}_\la)\ar@{=}^<>(0.5){\eqref{filt}}[r]&
\U\g^{\eta+\rho}/\ann[\g^{\eta+\rho};\, v_\la]\ar@{=}^<>(0.5){\Ad u}[r]&
\U\h/I(V_\la)=\C[\Specm V_\la].}
\eeq

Now, we combine the above isomorphisms for $\mu=o$ and
for $\mu=\rho$.
In this way, we obtain a  chain of graded algebra isomorphisms
$$
\U\g^e/\ann[\g^e;\, v_\la]=\U\g^\eta/\ann[\g^{\eta};\, v_\la]\cong
H^\hdot(\overline{\OO}_\la)\cong \ggr_\idot
H_\rho(\overline{\OO}_\la)=\ggr_\idot\C[\Specm V_\la].
$$

Thus, the  isomorphism on the right of \eqref{2iso} is provided by the 
composite.
\end{proof}

\begin{proof}[Proof of Proposition \ref{peterson2}] We use Proposition
\ref{HGGr} and Theorem \ref{sat} in the case $\mu=o$, cf. Remark \ref{mu0}.
Thus, we may define the required map $f_o$ as a composite of the following chain
of maps, cf.  \eqref{u2v2},
\begin{align*}V_\la=\IH^\hdot(\overline{\OO}_\la)=
H^\hdot(\overline{\OO}_\la,\
IC(\overline{\OO}_\la))&\too H^\hdot(\OO_\la,\
\jmath_\la^*IC(\overline{\OO}_\la))=H^\hdot(\OO_\la,\ \C_{\OO_\la})\\
&\en=
H^\hdot(\OO_\la)=H^\hdot(\Gv/P_\la)=\sym\h^{W_\la}/\m_o\cd\sym\h^{W_\la}.
\end{align*}

It remains to prove the surjectivity of the  composite map
\eqref{uuvv}. According to the above,  this composite may be identified
 with  the restriction
map  $H^\hdot(\overline{\OO}_\la)\to
H^\hdot(\OO_\la).$ The latter map
is surjective since the spaces in question have compatible
partitions into cells of even real dimension.
\end{proof}

\subsection{Proof of Theorem \ref{key}.}\label{pf_key} Part (i) is clear from 
the proof of Proposition \ref{peterson2} since,
for a minuscule weight $\mu$, each of the imbeddings
in \eqref{u2v2} is a bijection.

To prove (ii),  for any anti-dominant weigth $\la,$ write
$IC_\la:=IC(\overline{\OO}_\lambda).$ 
Let $D^b(\gr, \Tv)$ be the bounded  $\Tv$-equivariant derived category
of constructible complexes on $\gr$, with compact support, as defined in \cite{BL1}. For
any objects $\L,\L'\in D^b(\gr, \Tv),$   we define a family
of {\em specialized equivariant Ext-groups}
$$\Ext_\la(\L,\L'):=\Ext^\hdot_{D^b(\gr, \Tv)}(\L,\L')/\m_\la\cd
\Ext^\hdot_{D^b(\gr, \Tv)}(\L,\L'),\qquad\la\in \h^*.
$$

Let $\mu$ be a minuscule anti-dominant weight,
and let $j_\mu:\ \OO_\mu=\overline{\OO}_\mu\into\gr$ be
the corresponding closed imbedding. Clearly, we have
$IC_\mu=(j_\mu)_!\C_{\OO_\mu}$. 
Further, given $\la\in\h^*$, put 
$ J(\la):=\ann[\g^{\eta+\la};\, V_\mu]$.
From Lemma \ref{lem}
and the proof of Theorem \ref{peterson}, we deduce an isomorphism
$V_\mu=\U\g^{\eta+\la}/ J(\la).$

Let $\MM\in\pgr$ be a perverse sheaf
supported on the connected component of the
loop Grassmannian  that contains the orbit $\OO_\mu$,
and let $V:=\sat(\MM)\in\Rep\g$ be the
representation corresponding to $\MM$ via the Satake equivalence.
By  Theorem 1.10.3 of \cite{G1} (whose proof depends on
\cite{G2} in an essential way), for
any $\la\in\h^*$  there are
canonical  vector space isomorphisms
\begin{align}\label{ext}
H_\la(j_\mu^!\MM)=
 \Ext_\la\big((j_\mu)_!\C_{\OO_\mu},\  \MM\big)&=\xymatrix{
\Ext_\la(IC_\mu,\   \MM)\ar@{=}[rrr]^<>(0.5){\text{\cite{G1}, Thm. 1.10}}&&&
\Hom_{\g^{\eta+\la}}(V_\mu,\  
V)}\nonumber\\
&=\Hom_{\g^{\eta+\la}}(\U\g^{\eta+\la}/ J(\la),\   V)=
V^{ J(\la)}.
\end{align}

For $\la=0$, we have $\g^\eta=\g^e$, in particular,
$\ann[\g^e;\, v_\mu]= J(0)$. In this case, all the
isomorphisms in \eqref{ext} respect the gradings. 
On the other hand, we may take $\la=\rho$.
Then, according to \eqref{filtalg}, we get 
$$V_\mu=H_\rho(\gr,\ IC_\mu)=H_\rho(\OO_\mu)=\U\h/I(\Specm V_\mu).$$
Hence, the ideal $ J(\rho)$  is obtained from $I(\Specm V_\mu)$ via
the adjoint action
by the element $u$, cf. \eqref{filtalg}.  Thus, we obtain 
a graded, resp. filtered,  isomorphism 
\beq{resp}
H^\hdot(j_\mu^!\MM)\cong V^{\ann[\g^e;\, v_\mu]},
\quad\text{resp.}
\quad
H_\rho(j_\mu^!\MM)\cong V^{I(\Specm V_\mu)}.
\eeq

Observe that the complex $j_\mu^!\MM$ is pure, by pointwise purity
of any object of  category $\pgr$.
Hence, the isomorphism in \eqref{fixo} applies to the complex
$\L=j_\mu^!\MM$, and we deduce that
$H^\hdot(j_\mu^!\MM)\cong\ggr H_\rho(j_\mu^!\MM)$.
Further, the fixed point decomposition \eqref{fp_decomposition}
shows that $V^{I(\Specm V_\mu)}= \oplus_{\nu\in W(\mu)} V(\nu)$.
Hence, we compute $\dim V^{I(\Specm V_\mu)}=\dim V(\mu)\cdot |W(\mu)|$,
since $\dim V(\nu)=\dim V(\mu)$ for any $\nu\in W(\mu)$.
The proof is completed by equating dimensions of the various isomorphic
vector spaces
considered above,
\begin{align*}
\dim V^{\ann[\g^e;\, v_\mu]}=\dim H^\hdot(j_\mu^!\MM)&=\dim\ggr
H_\rho(j_\mu^!\MM)=H_\rho(j_\mu^!\MM)\\
&=
\dim V^{I(\Specm V_\mu)} =\dim V(\mu)\cdot |W/W_\mu|.
\end{align*}

\section{Kostant's theorem on $\Ad G$-invariant polynomials.}
 \subsection{} Let $G$ be the adjoint group of a semisimple
Lie algebra $\g$. The group $G$ acts on $\g^*$
by the (co)adjoint action.
In this section, we discuss various incarnations of the
 famous result, due to Kostant \cite{Ko1}, concerning the structure of the
 {\em coadjoint quotient} map
$\varpi:\ \g^*\onto \g^*/\!/\Ad G.$

The most essential part of  Kostant's result reads

\begin{theorem}[Kostant]\label{orig}
The algebra $\sym\g$ is a free graded $\sym\g ^G $-module;
furthermore,
the quotient algebra $\sym\g/\m\cd\sym\g$ is a normal domain, for any maximal ideal
$\m\sset \sym\g ^G $.
\end{theorem}

A key geometric result in the original approach of 
\cite{Ko1} was the normality of the {\em nilpotent variety} in $\g$.
 Kostant's 
proof of that result involved certain properties of Cohen-Macauley
rings which boil down, essentially, to the {\em Serre normality
criterion}, see \cite{Se}. A simplified  argument, that bypasses
references to Serre's criterion, was later found by Bernstein
and Lunts \cite{BL2}, cf. also \cite{CG}, \S6.7 for an exposition.
Below, we discuss an alternate approach to Theorem \ref{orig}
which is
of a more `topological' nature.

To explain our approach, it will be convenient to reformulate
Kostant's result. To this end, 
let $\B$ be the flag variety for $G$, viewed as the
 variety
of Borel subalgebras in $\g$. 
One introduces an incidence variety $\wg^*:= \{(x, \b)\in \g^*\times\B\;\big|\;
x|_{[\b,\b]}=0\}\,$ that fits into  the following
commutative diagram,
\begin{equation}\label{gro}
\xymatrix{
&\g^*\ar[dl]_<>(0.5){\varpi}&&\wg^*\;\ar[ll]_<>(0.5){\mu}\ar[dl]_<>(0.5){\nu}
\ar[dr]^<>(0.5){p}\ar@{^{(}->}[rr]^<>(0.5){\epsilon}&&\g^*\times\B\ar[dl]^<>(0.5){pr_2}\\
{\h}^*/W&&\h^*\ar[ll]_<>(0.5){\pi}&&\B&
}
\end{equation}

In this diagram,
$\pi$ is a finite flat morphism, since $\C[{\h}]$ is free
over $\C[{\h}]^W$,  and the map $\nu: x\mto x|_{\b/[\b,\b]}$ is a smooth  morphism
induced by restriction via the diagram
$\g\hookleftarrow\b\onto\b/[\b,\b]=:\h,\ \forall\b\in\B,$ cf. eg. 
\cite{CG}, formula (7.3.1).
Further, the   map $\mu$, resp. $p=pr_2\ccirc\epsilon $, is induced by the first,
resp.   second,  projection of $\g\times\B$.

The commutative parallelogram on the left of \eqref{gro}
 is known as the {\em  Grothendieck-Springer resolution},
 cf. eg.  \cite[\S3.2]{CG}. This  parallelogram is {\em not}
a cartesian square; in fact, the map $\mu\times\nu$ may be factored as
 follows
\beq{gh}
\xymatrix{
\wg^*\en
\ar@/^1.5pc/[rrrr]|-{\,\mu\times\nu\,}
\ar@{->>}[rr]_<>(0.5){\psi}&&\en\g^*\times_{\hw}\h^*\en\ar@{^{(}->}[rr]&&
 \en\g^*\times\h^*.
}
\eeq
Here, the map $\psi$ is dominant and proper, hence,  surjective.
Further, by definition, we have $\C[\g^*\times_{\hw}\h^*]=\sym\g\o_{\sym\h^W}\sym\h.$

The following version of Kostant's result is often more convenient,
cf. eg. \cite{BBM}.
\begin{theorem}\label{mil} The morphism \eqref{gh}
induces an  algebra
isomorphism 
$$\psi^*:\ \sym\g\o_{\sym\h^W}\sym\h\iso\Gamma(\wg^*,\oo_{\wg^*}).$$

Furthermore, we have $H^i({\wg^*},\oo_{\wg^*})=0,$ for any $i>0$.
\end{theorem}

\begin{proof}[Theorem \ref{mil} implies  Theorem \ref{orig}]
First of all, we recall that any
nonnegatively graded flat  $\sym\g ^G $-module is necessarily free.
Therefore, to prove the first statement of  Theorem \ref{orig},
it suffices to show that $\sym\g$ is a  flat  $\sym\g ^G $-module.
But this follows from  Theorem \ref{mil} by a flat base change.
In fact, the classical result says that
 $\sym\h$ is a free $\sym\h^W$-module of rank $|W|$.
Hence,  $\sym\g$ is flat over  $\sym\g ^G $
if and only if  $\sym\g\o_{\sym\h^W}\sym\h$ is flat over  $\sym\h$,
by the  base change. Applying Theorem \ref{mil}, we see that
it suffices to prove that $\Gamma({\wg^*},\oo_{\wg^*})$
is  flat over  $\sym\h$. But this is clear since 
 $\nu$  is a smooth morphism, hence, it is flat.

The second  statement of  Theorem \ref{orig} is proved similarly
by restricting diagram  \eqref{gro} to a fiber
of the map
$\nu$, which is a smooth morphism. Specifically, given a maximal ideal $\m\sset \sym\g ^G =
\sym\h^W$, choose a maximal ideal of the algebra $\sym\h=\C[\h^*]$ over $\m$, i.e. an element $\chi\in\h^*$
 such that, for the  maximal ideal $(\chi)\sset
\C[\h^*]$  associated with $\chi$, one has
$(\chi)\cap \sym\h^W=\m$.
We set  $\C_\chi=\sym\h/(\chi),$ resp. $\C_\m=\sym\h^W/\m.$
Also, let $\wgc:=\nu\inv(\chi)$ be the fiber of the map $\nu$
over the  point $\chi\in\h^*$.

Then, from  diagram  \eqref{gro}, we get
$\Gamma({\wgc},\ \oo_{\wgc})=\Gamma(\wg^*, \oo_{\wg^*})\o_{\sym\h}\C_\chi$.
Hence, using  Theorem \ref{mil} we find
$$
\Gamma({\wgc}, \oo_{\wgc})=\big(\sym\g\o_{\sym\h^W}\sym\h\big)\o_{\sym\h}\C_\chi
=\sym\g\o_{\sym\h^W}\C_\m=\sym\g/\m\cd\sym\g.
$$
The scheme ${\wgc}$ being a smooth connected variety, we conclude
that the algebra $\Gamma({\wgc},\ \oo_{\wgc})$ is a normal domain,
and we are done.
\end{proof}

\subsection{A `topological' proof of  Theorem \ref{mil}.} 
Given  a coherent sheaf $\mathcal V$,
we write  $\Sym \mathcal V$ for the
corresponding symmetric algebra, a quasi-coherent sheaf.

Write $\g_\B:=\g\o\oo_\B$ for the trivial
sheaf on $\B$ with fiber $\g$, and  $\nf_\B\sset \g_\B$ for a ({\em nontrivial}) subsheaf
whose fiber
at each point $\b\in\B$ is $\nf=[\b,\b]$, the nilradical
 of the Borel subagebra $\b$.
There is a canonical isomorphism
\beq{wgwg}
\nf_\B\cong T^*_\B:=\text{cotangent
sheaf on } \B.
\eeq
 
The closed imbedding
$\epsilon,$  in diagram \eqref{gro},  makes $p: {\wg^*}\to\B$ a  sub-vector bundle of
 the trivial bundle $pr_2: \g^*\times\B\to\B$.
It is clear that the sheaf of sections of the vector bundle dual to  ${\wg^*}$
is isomorphic naturally to $\g_\B/\nf_\B$, a quotient sheaf.
The vector bundle projection $p$   being affine,
 we get $H^i({\wg^*},\ \oo_{\wg^*})=H^i(\B,\ p_\idot\oo_{\wg^*})=
H^i(\B,\  \Sym(\g_\B/\nf_\B) ).$

The computation of the rightmost sheaf cohomology group 
given below is based on
 an old  idea
that goes back to Jantzen and  Mili\v ci\v c.

For each $q=0,1,\ldots,$ let
${\mathcal
K}^q:=(\Sym\g_\B)\bigotimes_{\oo_\B}(\wedge^q\nf_\B)=
\sym\g\ \otimes\ \wedge^q\nf_\B$,  a quasi-coherent sheaf on
$\B$. The sheaf imbedding $\nf_\B\into\g_\B$ gives, in a
standard way, a  Koszul differential
 ${\mathcal K}^\hdot\to{\mathcal K}^{\hdot-1}$. This  differential is a
sheaf morphism that makes ${\mathcal K}^\hdot$ a resolution of the sheaf $\Sym(\g_\B/\nf_\B) $.
Further, using  \eqref{wgwg}, one can write
 $\wedge^\hdot\nf_\B= \wedge^\hdot T^*_\B=\Om^\hdot_\B,$ the sheaf of algebraic
 differential forms
on $\B$.
Thus, we obtain ${\mathcal K}^\hdot=\sym\g\o\Om^\hdot_\B$.

The  resolution ${\mathcal K}^\hdot\to\Sym(\g_\B/\nf_\B) \to0$ may be used for computing
 sheaf cohomology. Specifically, there is  a convergent third quadrant 
spectral sequence 
$$
E_1^{p,-q}=H^p(\B,\ \sym\g\o\Om^q_\B)
\en\Rightarrow\en
E_\infty^{p,-q}=\ggr_pH^{p-q}\big(\B,\ \Sym(\g_\B/\nf_\B)\big)=
\ggr_pH^{p-q}({\wg^*},\ \oo_{\wg^*}).
$$

We have $H^p(\B,\ \sym\g\o\Om^q_\B)=\sym\g\o
H^p(\B,\ \Om^q_\B)=\sym\g\o H^{p,q}(\B)$,
by the Hodge decomposition for the de Rham cohomology of
a K\"ahler
manifold.
Recall that, for the flag manifold $\B$,
the group $H^{p,q}(\B)$  vanishes unless $p=q$.
 We conclude that the spectral sequence degenerates
at the $E_1$-term. Moreover, we deduce that the cohomology group
$\ggr H^i({\wg^*},\ \oo_{\wg^*})$ vanishes for any $i>0$, and there
are graded algebra isomorphisms
\beq{hilb}
\ggr H^0({\wg^*},\ \oo_{\wg^*})= E_\infty=
E_1=\sym\g\o H^\hdot(\B).
\eeq

Here, the grading on the leftmost term is the one that comes from the
$\C^\times$-action on $\wg^*$ by dilations along the fibers, and the grading on
the tensor product on the right is the tensor product  grading
induced by the natural grading on a symmetric algebra
and the grading on $H^\hdot(\B)$ such that $H^{2i}(\B)$
is placed in degree $i$. Thus, we have proved that
$H^i({\wg^*},\ \oo_{\wg^*})=0, \ \forall i>0.$

To complete the proof, we use the Borel isomorphism and write $\sym\h$, 
a free graded $\sym\h^W$-module, in the
form
$\sym\h=\sym\h^W\o H^\hdot(\B)$, where the cohomology of
the flag manifold is equipped with the same grading as above.

Further,
view the map
$\psi^*$, in the statement of the theorem, as a map of 
$\C[\g^*]$-modules. 
The kernel of this map must be a torsion $\C[\g^*]$-submodule since the
morphism \eqref{gh} is easily seen to be an isomorphism
at the generic point. On the other hand, we have
\beq{hilb2}
\sym\g\o_{\sym\h^W}\sym\h=\sym\g\o_{\sym\h^W}\big(\sym\h^W\o H^\hdot(\B)\big)=
\C[\g^*]\o H^\hdot(\B),
\eeq
 is a finite rank free  $\C[\g^*]$-module.
Therefore, this   $\C[\g^*]$-module has no  torsion submodules
and we conclude that the map $\psi^*$ is injective.

To prove surjectivity, we 
let the group $\C^\times$ act
on the  vector space $\g\times\h$, resp. on ${\wg^*}$, by
dilations.
This makes the map $\psi:\ {\wg^*}\to\g^*\otimes_{\hw}\h^*$
 a $\C^\times$-equivariant morphism
of $\C^\times$-varieties.
It follows that
$\psi^*$, the pull-back  map,  is an injective degree preserving 
map of the corresponding coordinate rings.
Comparing formulas \eqref{hilb} and \eqref{hilb2} we see
that, for each degree, the corresponding homogeneous components of  $\sym\g\o_{\sym\h^W}\sym\h$
and $\Gamma({\wg^*},\ \oo_{\wg^*})$ have equal dimensions.
Thus, the map $\psi^*$ must be an isomorphism.
\qed

\begin{remarks} \vi
Let $\g^{\rs}\subset \g^*$ be the union
of   $\Ad G$-conjugacy classes  in $\g^*$ of maximal dimension, and 
 put $\wg^{\rs} :=
\mu^{-1}({\g}^{\rs})$.  We claim that the restriction of the  morphism
$\psi$ in \eqref{gh} to
$\wg^{\rs}$
yields
 an isomorphism of algebraic varieties
$$
\psi^{\rs}:\
 \wg^{\rs} \stackrel{\sim}{\longrightarrow}
 \g^{\rs}\times_{_{{\h^*}/W}} {\h^*}\,.
$$

To see this, one verifies first, using Jordan decomposition in $\g$,
that  the map $\,\mu^{\rs}: \wg^{\rs} \to {\g}^{\rs}\,$
 has finite fibers. It follows, since ${\h^*}$ 
is finite  over ${\h^*}/W$,  that  the map $\psi^{\rs}$  has
finite fibers as well. Furthermore, using that $\psi^{\rs}$ is also proper,
we deduce that it is, in fact, a finite dominant  morphism between smooth schemes.
Moreover, it is clear that  the map $\psi^{\rs}$ is one-to-one over
the Zariski open subset of {\em semisimple} regular elements of $\g$.
Therefore, such a morphism must automatically be  an isomorphism.\qed
\smallskip

\vii  Kostant showed that $\varpi^{-1}(0)$,
the scheme-theoretic zero fiber of the morphism $\varpi$, is reduced at any point of
 $\varpi^{-1}(0)\cap\g^{\rs}$.
To prove this, he   verified 
by a direct
computation
that the generators of $\cg^G$ have linearly independent
differentials
at any point of $\g^{\rs}$,
see  \cite{Ko2}, and also \cite[\S6.7]{CG} for a slightly different argument. 

Here is an alternate argument
which involves no computation and is independent of \cite{Ko2}.
By Remark (i), we may identify 
$\nu: \wg^{\rs} \to {\h^*}$, the restriction of the smooth morphism $\nu$
to $\g^{\rs}$, with the projection
$\g^{\rs}\times_{_{{\h^*}/W}} {\h^*}\to$
${\h^*}$. Applying to this projection 
base change with respect to the flat
map ${\h^*}\to{\h^*}/W$ in diagram
(\ref{gro}), we deduce that the morphism
$\varpi: \g^{\rs} \to {\h^*}/W$ is also smooth.
Hence its zero-fiber is reduced.\qed
\end{remarks}

\footnotesize{

}

\noindent
\footnotesize{\hphantom{x}\ab
 Department of Mathematics, University of Chicago, 
Chicago IL
60637, USA; 

\noindent
\hphantom{x}\ab ${\tt{ginzburg@math.uchicago.edu}}$

\end{document}

Let g be a complex semisimple Lie algebra and let G' be the Langlands 
dual group. We give a description of the cohomology algebra of an 
arbitrary spherical Schubert variety in the loop Grassmannian for G' as
a quotient of the form Sym(g^e)/I. Here, I is an appropriate ideal in
the symmetric algebra of g^e, the centralizer of a principal nilpotent in g.

We also discuss a `topological' proof of Kostant's famous result 
on the structure of the polynomial algebra on g.
\\

arXiv:0710.1443

psswd:  h8ri8

 supercedes arXiv:math.AG/9803141
vancouver